\documentclass[1p,preprint,number]{elsarticle}
\usepackage{setspace}

\usepackage{amsmath}
\usepackage{amssymb}
\usepackage{amsthm}
\usepackage{thmtools}
\usepackage[hidelinks]{hyperref}
\usepackage[nameinlink]{cleveref}
\usepackage{autonum} 

\crefformat{equation}{(#2#1#3)} 
\crefname{appendix}{}{} 

\usepackage{color}
\usepackage{graphicx}

\newcommand{\reals}{\ensuremath{\mathbb{R}}}
\newcommand{\norm}[1]{\ensuremath{\vert\vert#1\vert\vert}}
\newcommand{\pdiff}[1]{\ensuremath{\frac{\partial}{\partial #1}}}

\newcommand{\Pdiff}[2]{\ensuremath{\frac{\partial #1}{\partial #2}}}
\newcommand{\Pdifftwo}[2]{\ensuremath{\frac{\partial^2 #1}{\partial {#2}^2}}}

\newcommand{\cdiff}[1]{\ensuremath{\frac{D}{\partial #1}}}
\newcommand{\cdifftwo}[1]{\ensuremath{\frac{D^2}{\partial #1^2}}}
\newcommand{\vbar}{\ensuremath{\biggr\vert}}
\newcommand{\flow}{\ensuremath{\phi}}
\newcommand{\oT}[1]{\ensuremath{\mathring{T}_1 #1}}
\newcommand{\T}[1]{\ensuremath{T_\partial #1}}
\newcommand{\uT}[1]{\ensuremath{T_1 #1}}
\newcommand{\genflow}{\ensuremath{\mathcal F}}
\renewcommand{\O}{\ensuremath{\mathcal{O}}}

\newcommand{\comment}[1]{}

\newtheorem{theorem}{Theorem}
\newtheorem{lemma}[theorem]{Lemma}
\newtheorem{proposition}[theorem]{Proposition}
\newtheorem{corollary}{Corollary}[theorem]

\begin{document}
\begin{frontmatter}
	\title{
Travelling Times in Scattering by Obstacles in Curved Space}
\author{Tal Gurfinkel\corref{cor1}}
\ead{21710178@student.uwa.edu.au}
\author{Lyle Noakes\corref{}}
\ead{lyle.noakes@uwa.edu.au}
\author{Luchezar Stoyanov}
\ead{luchezar.stoyanov@uwa.edu.au}
\address{School of Mathematics and Statistics, University of Western Australia, Crawley 6009 WA, Australia}
\cortext[cor1]{Corresponding author}
\date{\today}
\begin{abstract}
	We consider travelling times of billiard trajectories in the exterior of an obstacle K on a two-dimensional Riemannian manifold $M$. We prove that given two obstacles with almost the same travelling times, the generalised geodesic flows on the non-trapping parts of their respective phase-spaces will have a time-preserving conjugacy. Moreover, if $M$ has non-positive sectional curvature we prove that if $K$ and $L$ are two obstacles with strictly convex boundaries and almost the same travelling times then $K$ and $L$ are identical.
\end{abstract}
\end{frontmatter}
\section{Introduction}

Let $M$ be a geodesically complete, 2-dimensional Riemannian manifold and let $K$ be a smooth codimension-0 submanifold of $M$ with boundary, such that $\overline{M\backslash K}$ is connected. Suppose there is another codimension-0 submanifold $S$ whose boundary is strictly convex, such that $K\subset S$. We will call $K$ an obstacle and $S$ a bounding submanifold. A generalised geodesic $\gamma:[a,b]\to S_K$ in $S_K = \overline{S\backslash K}$ is any unit speed, piecewise smooth extremal of the length functional
\[\int_a^b \norm{\dot\gamma}\ dt,\]
with respect to variations of paths fixing the endpoints (\cite{MORSE},\cite{LEERIEMANN}). That is, $\gamma$ is a piecewise smooth curve made up of smooth geodesic segments in $S_K$, which reflects off the boundary $\partial K$ symmetrically across the normal. More precisely still, there are discrete points $t_1,t_2,\dots$ where $\gamma$ is not differentiable, and there we have
	\begin{equation}\label{eq:symmetrytangent}
		\left<\dot\gamma(t_i^-),v\right>=\left<\dot\gamma(t_i^+),v\right>
	\end{equation}
with respect to any tangent $v$ to $\partial K$.
If $\gamma$ is a geodesic between two points $x,y\in\partial S$ then we say that $\gamma$ is an $(x,y)$-geodesic. We define the set of travelling times $\mathcal T_K$ of $K$ to be the set of all triples $(x,y,t_\gamma)$ where $t_\gamma$ is the length of an $(x,y)$-geodesic. We also call $t_\gamma$ the travelling time of $\gamma$.

Let $\uT S_K$ be the unit tangent bundle of $S_K$ and define the quotient $\oT S_K$ by identifying angles with their reflection across the boundary of $K$, according to \cref{eq:symmetrytangent}.
	Let $\gamma$ be a generalised geodesic in $S_K$ generated by a point $(x,\omega)\in \uT{S_K}$. We say that $\gamma$ is \emph{non-trapped} if there are distinct $t_0,t_0'\in\reals$ such that $\gamma(t_0),\gamma(t_0')\in \partial S$. Otherwise we say that that $\gamma$ is \emph{trapped}. 
	Denote the set of all $(x,\omega)\in \uT{S_K}$ which generate a trapped generalised geodesic by $Trap(S_K)$. Also let \begin{equation}
		Trap^\partial(S_K) = \{ (x,\omega)\in Trap(S_K) : x\in \partial S \}
	\end{equation}
	We will define a generalisation of the geodesic flow as follows.
	Let $\gamma_{(x,\omega)}$ be the unique generalised geodesic in $S_K$ defined by the initial conditions
	\begin{align}
		\gamma_{(x,\omega)}(0) = x && \dot\gamma_{(x,\omega)}(0) = \omega
	\end{align}
	Define for each $t\in\reals$ the generalised geodesic flow, $\genflow_t:\oT{S_K}\to \oT{S_K}$ by
	\begin{equation}
		\genflow_t(x,\omega) = (\gamma_{(x,\omega)}(t),\dot\gamma_{(x,\omega)}(t))
	\end{equation}
	This is the billiard flow as defined in \cite{MR832433} on general Riemannian manifolds.
	Note that if $K$ is empty then $\genflow_t$ is simply the geodesic flow (see e.g. \cite{GEOFLOW}), which we denote by  $\flow_t:\uT S_K\to \uT S_K$.
	Let $K$ and $L$ be two obstacles with the same bounding manifold $S\subseteq M$. $K$ and $L$ are said to have \emph{conjugate flows} if there exists a homeomorphism
	\begin{equation}
		\Phi : \oT{S_K}\backslash Trap(S_K)\to \oT{S_L}\backslash Trap(S_L)
	\end{equation}
	Such that $\Phi|_{\T{S}} = \text{Id}$ and
	\begin{equation}\label{eq:conjugacy}
		\genflow^{(L)}_t \circ \Phi = \Phi\circ \genflow^{(K)}_t\text{ for all } t\in\reals
	\end{equation}
 Moreover, let
 		$T_K(x,y) = \{t\in [0,\infty) : (x,y,t)\in \mathcal T_K\}$.
 	We say that $K$ and $L$ have \emph{almost the same travelling times} if $T_K(x,y) = T_L(x,y)$ for almost all $(x,y)\in \partial S\times\partial S$.
 
We are now ready to state the two main results of this paper.
\begin{theorem}[Conjugacy Theorem]\label{thm:conjugacy}
	Two obstacles $K$ and $L$ with the same bounding manifold $S\subseteq M$ have conjugate flows if and only if they have almost the same travelling times.
\end{theorem}
\begin{theorem}[Uniqueness of Convex Obstacles]\label{thm:uniqueness}
	Suppose that $M$ has non-positive sectional curvature. If $K$ and $L$ are two (disjoint unions of) strictly convex obstacles with almost the same travelling times in $M$, then $K = L$.
\end{theorem}
Inverse problems related to  metric rigidity have been studied for a very long time in Riemannian geometry - see e.g. \cite{MR3454376} and the references there for more information. A different kind of problems studied extensively recently for various types of dynamical systems concern the so called Marked Length Spectrum, defined as the set of all lengths of periodic orbits in phase space together with their marking - see \cite{MR3990606}, \cite{MR3743701} and the references there for more information. Various similar problems have been considered in scattering by obstacles in Euclidean spaces in the last 20 years. A natural and rather important problem in inverse scattering by obstacles in $\reals^n$ is to get information  about the obstacle $K$ from its so-called scattering length spectrum, which is in a certain way related to travelling times of billiard (and more general) trajectories in the exterior of the obstacle - see \cite{MR1942248} for details. An analogue of the Conjugacy Theorem (\Cref{thm:conjugacy}) above was first established in \cite{MR1942248} in this context. For travelling times of trajectories in the exterior of obstacles in $\reals^n$, such a theorem was established in \cite{MR3351977}. In both \cite{MR1942248} and \cite{MR3351977} the Conjugacy Theorem was used to recover geometric information about the obstacle from travelling times. 

It turned out that some kinds of obstacles are uniquely recoverable from their travelling times (and also from their scattering length spectra), e.g. star-shaped obstacles are in this class, as shown in \cite{MR1942248}.  Obstacles in $\reals^n$  that are disjoint unions of strictly convex bodies with smooth boundaries are also uniquely recoverable - this was proved in \cite{MR3359579} for $n \geq 3$ and in \cite{NSUniqueness} for $n =2$. In \cite{MR3655806} a certain generalisation was established of the well-known Santalo's Formula in Riemannian geometry. As a consequence, it was shown in \cite{MR3655806} that, assuming the set of trapped points has Lebesgue measure zero, one can recover for example the volume of the obstacle from travelling times. 

It should be remarked that in general, the set of trapped points could be rather large. As an example of M. Livshits shows  (see e.g. Figure 1 in \cite{MR3359579}  or  in \cite{MR3655806}), in some cases the set of trapped points contains a non-trivial open set, and then the  obstacle cannot be recovered from travelling times. \Cref{thm:uniqueness} above establishes a result similar to the one in \cite{NSUniqueness}, although the situation considered in this paper is significantly more complicated.

This paper is separated into a preliminary section, three main sections and an appendix. In \Cref{sec:conjugacy,sec:negcurve}, we prove \Cref{thm:conjugacy,thm:uniqueness} respectively. While in \Cref{sec:hardproofs} we give proofs for three technical propositions which are of fundamental importance to proving \Cref{thm:uniqueness}. Throughout the paper we draw on arguments from \cite{MR3351977} and \cite{NSUniqueness}, although in general we either adapt and extend the arguments to the more complicated case of Riemannian 2-manifolds or provide completely new proofs.

\section{Preliminaries}
We now state some results which will be useful in proving \Cref{thm:conjugacy}. The following result is well known, see \cite{MR872698} for a proof.
\begin{lemma}\label{thm:tangency}
	For almost all $(x,\omega)\in \T{S}\backslash Trap^\partial(S_K)$ the generalised geodesic defined by $\gamma(t) = \genflow_t(x,\omega)$ is not tangent to $\partial S_K$ anywhere.
\end{lemma}

\begin{lemma}\label{thm:distincttimes}
	Fix $x_0\in \partial S$. The set of pairs of distinct directions $\omega_1,\omega_2\in \T{S}_{x_0}$ which generate generalised geodesics with  the same endpoint and the same travelling time is countable.
\end{lemma}
\begin{proof}
	See \Cref{appendix}.
\end{proof}

The proof of the following fact would use the exact same argument as proposition 2.3 in \cite{MR1775186}. We therefore omit the proof.

\begin{lemma}\label{thm:trappedboundary}
	The set $Trap^\partial(S_K)$ has measure zero in $\T{S}$.
\end{lemma}

\begin{lemma}\label{thm:xydependance}
 	Suppose $\gamma$ is generalised geodesic starting at $x_0$ and ending at $y_0$, with $k$ successive reflection points $x_1,\dots, x_k = y_0$. Let $\omega_0$ be the initial velocity of $\gamma$. Then there exist neighbourhoods $W$ of $(x_0,y_0)$ and $U_i\subseteq \partial S_K$ of $x_i$ such that there exist unique smooth maps $x_i:W\to U_i$, with $x_i(x_0,y_0) = x_i$. Furthermore, there is a neighbourhood $V$ of $(x_0,\omega_0)$ such that $F:V\to W$ defined by
 	\begin{equation}
 		F(x,w) = (x, \pi_1\circ\mathcal{P}_K(x,\omega))
 	\end{equation}
 	is a diffeomorphism onto $W$.
 \end{lemma}
 See e.g. \cite{MR3351977} for a proof.

\section{The Conjugacy Theorem}\label{sec:conjugacy}
In this section we first prove two useful results, \Cref{thm:timeangle} and \Cref{thm:flowequivalency}. Then we finally give a proof of \Cref{thm:conjugacy}.
\begin{lemma}\label{thm:timeangle}
	Suppose that $\gamma$ is a non-trapped generalised geodesic in $S_K$ from $x\in \partial S$ to $y\in\partial S$. Then $grad_x T = -\dot\gamma(t_0)/\norm{\dot\gamma(t_0)}$, where $T(x,y)$ is the length of the geodesic $\gamma$.
\end{lemma}
\begin{proof}
	Suppose $\gamma:[a,b]\to S_K$ is split into geodesic segments $\gamma|_{[t_{i-1},t_i]}$ for \[a = t_0<t_1<\dots<t_{n-1}<t_n = b.\] Let $\gamma_h$ be any variation of $\gamma$ fixing the endpoint, $\gamma_h(t_n) = \gamma(t_n)= y$. Reparameterise each $\gamma_h$ so that $t_n(h) = b$ for all $h$.  Consider the derivative of the travelling time function $T_h=\int_{t_0}^{t_n} \norm{\dot\gamma_h}\ dt$, as follows:

	\begin{align}
		\pdiff{h}T_h &= \pdiff{h}\sum_{i=1}^n\int_{t_i-1}^{t_i}\norm{\dot\gamma_h}\ dt\\
		&= \sum_{i=1}^n\left(\norm{\dot\gamma_h(t_i)}\pdiff{h}t_i-\norm{\dot\gamma_h(t_{i-1})}\pdiff{h}t_{i-1} + \int_{t_i-1}^{t_i}\pdiff{h}\norm{\dot\gamma_h} \right)\label{eq:fulltimesum}
	\end{align}
	Then consider the following sum from \cref{eq:fulltimesum}
	\begin{equation}\label{eq:timeparts}
		\sum_{i=1}^n\left(\norm{\dot\gamma_h(t_i)}\pdiff{h}t_i-\norm{\dot\gamma_h(t_{i-1})}\pdiff{h}t_{i-1}\right)
	\end{equation}
	Since $\norm{\dot\gamma_h}$ is constant with respect to $t$, and \[\pdiff{h}t_0=\pdiff{h}t_n=0,\] \cref{eq:timeparts} sums to 0 as well. Thus we can continue from \cref{eq:fulltimesum} as follows
	\begin{align}
		\pdiff{h}T_h&=\sum_{i=1}^n\int_{t_{i-1}}^{t_i}\frac{1}{\norm{\dot\gamma_h}}\left<\pdiff{h}\dot\gamma_h,\dot\gamma_h\right>\ dt \\
		&= \frac{1}{\norm{\dot\gamma_h}}\sum_{i=1}^n\left<\pdiff{h}\gamma_h,\dot\gamma_h\right>\vbar^{t_i}_{t_{i-1}} -\int_{t_{i-1}}^{t_i}\left<\pdiff{h	}\gamma_h,\cdiff{t}\dot\gamma_h\right> \ dt\label{eq:finalsum}
	\end{align}
	Evaluating at $h=0$, the first terms cancel since $\gamma(t_i)$ satisfy \cref{eq:symmetrytangent} for all $1\leq i<n$, and $v\in T_{\gamma(t_i)}\partial K$. The second terms are 0 since $\gamma$ is a geodesic on each segment. Thus we find
	\begin{equation}
		\pdiff{h} T_h|_{h=0}=-\left<\frac{\dot\gamma(t_0)}{\norm{\dot\gamma(t_0)}},\pdiff{h}\gamma_h|_{h=0}\right>
	\end{equation}
	For any variation which keeps the endpoint fixed. That is,  
	\begin{equation}
		grad_x T = -\dot\gamma(t_0)/\norm{\dot\gamma(t_0)}.
	\end{equation}
\end{proof}

\begin{theorem}\label{thm:flowequivalency}
	Suppose $K$ and $L$ are two obstacles with almost the same travelling times. Let $\genflow^{(K)}_t$ and $\genflow^{(L)}_t$ be their flows. Then $\genflow^{(K)}_t = \genflow^{(L)}_t$ on $\T{S}$.
	That is, if $(x,\omega)\in \T{S}$ and $t_0\in\reals$ are such that $\genflow^{(K)}_{t_0}(x,\omega)\in \T{S}$ then
	\begin{equation}
		\genflow^{(K)}_{t_0}(x,\omega) = \genflow^{(L)}_{t_0}(x,\omega)
	\end{equation}
\end{theorem}

\begin{proof}
	Since $K$ and $L$ have almost the same travelling times, there is a subset $R\subseteq \partial S\times \partial S$ such that 
	\begin{equation}
 		T_K(x,y) = T_L(x,y)\text{ for all }(x,y)\in R
 	\end{equation}
 	And $(\partial S\times \partial S)\backslash R$ is a set of measure 0. By using \Cref{thm:tangency} we can assume that if $\gamma$ is an $(x,y)$-geodesic with $(x,y)\in R$ then $\gamma$ is not tangent to $K$ or $L$ at any point. Furthermore, by \Cref{thm:distincttimes} we can assume that if $\gamma$ and $\delta$ are $(x,y)$-geodesics with $(x,y)\in R$ then $\gamma$ and $\delta$ either have distinct travelling times, $t_\gamma\neq t_\delta$,  or $\gamma = \delta$.
 	
 	Let 
 	\begin{equation}
 		\Omega = \T{S}\backslash (Trap^\partial(S_K)\cup Trap^\partial (S_L))
 	\end{equation}
 	That is, the set of points along the boundary $\partial S$ which generate non trapped geodesics in both $S_K$ and $S_L$. By \Cref{thm:trappedboundary}, $\Omega$ is dense in $\T{S}$, and since $\genflow^{(K)}_t$ and $\genflow^{(L)}_t$ are continuous it suffices to show that they are equivalent on $\Omega$. So pick any $(\widetilde x_0,\widetilde \omega_0)\in \Omega$ and let $t_0$ be the travelling time of the generalised geodesic in $K$ generated by $(\widetilde x_0,\widetilde \omega_0)$. Also, let $(\widetilde y_0,\widetilde \rho_0) = \genflow^{(K)}_{t_0}(\widetilde x_0,\widetilde \omega_0).$ Note that $(\widetilde x_0,\widetilde y_0)$ might not be in $R$. However, by \Cref{thm:xydependance} there is a diffeomorphism $F:V\to W$ between neighbourhoods of $(\widetilde x_0,\widetilde y_0)$ and $(\widetilde x_0,\widetilde \omega_0)$ respectively. Since $R$ is dense in $\partial S\times \partial S$, we can find $(x_0,\omega_0)$ arbitrarily close to $(\widetilde x_0,\widetilde \omega_0)$ such that $(x_0,y_0) = F(x_0,\omega_0)$ is in $R$.
 	
 	We now proceed to show that any $(x_0,y_0)$-geodesic in  $K$ and $L$ with the same travelling time will have the same incoming and outgoing angles $\omega_0$ and $\rho_0$. Therefore, in the limit we will have \[\genflow^{(K)}_{t_0}(\widetilde x_0,\widetilde \omega_0) = \genflow^{(L)}_{t_0}(\widetilde x_0,\widetilde \omega_0).\]
 	Let $\gamma_K$ be the geodesic generated by $(x_0,\omega_0)$ in $S_K$. Let $t_K$ be its travelling time. Now since $T_K(x_0,y_0)=T_L(x_0,y_0)$, there is an $(x_0,y_0)$-geodesic $\gamma_L$ in $L$ with the same travelling time $t_L = t_K$.
 	To show that $\gamma_L$ has the same incoming and outgoing directions as $\gamma_K$ we make use of \Cref{thm:xydependance}. This gives us neighbourhoods $W$ of $(x_0,y_0)$ and $U_i^{(K)}$ of $x_i^{(K)}$, the reflection points of $\gamma_K$. We also then have the diffeomorphism $F:V\to W$ as before, and unique smooth maps $x_i^{(K)}:W\to U_i^{(K)}$ such that $x_i^{(K)}(x,y)$ is the $i^{\text{th}}$ reflection point of the generalised geodesic $\gamma_K(x,y)$ in $S_K$ generated by ${F^{-1}(x,y)\in V}$. Similarly we have neighbourhoods $U_i^{(L)}$ and maps $x_i^{(L)}:W\to U_i^{(L)}$ for the generalised geodesics $\gamma_L(x,y)$ in $S_L$. Note that $\gamma_K(x_0,y_0)=\gamma_K$ by definition and similarly $\gamma_L(x_0,y_0)=\gamma_L$. It's important to remember that $\gamma_L$ and $\gamma_K$ are possibly not the same generalised geodesic.
 	
 	We want to show that if we shrink $W$ enough, then $\gamma_K(x,y)$ and $\gamma_L(x,y)$ will have the same travelling times for all $(x,y)\in R\cap W$. Let $t_K(x,y)$ and $t_L(x,y)$ be the travelling times of $\gamma_K(x,y)$ and $\gamma_L(x,y)$ respectively. Once again, for every $(x,y)\in R\cap W$ we have
 	\begin{equation}\label{eq:timeequality}
 		T_K(x,y)=T_L(x,y)
 	\end{equation} 
 	So for each $(x,y)\in R\cap W$ there exists a unique $(x,y)$-geodesic $\widetilde \gamma_L(x,y)$ in $S_L$ such that its travelling time 
 	\begin{equation}\label{eq:timeequality2}
 		\widetilde t_L(x,y) = t_K(x,y).
 	\end{equation} 
 	Now we will find that $\gamma_L(x,y) = \widetilde\gamma_L(x,y)$ for all $(x,y)\in R\cap W$ provided we make $W$ small enough. Suppose that is not the case. Let \[W = W_1 \supset W_2 \supset \cdots\] be a decreasing sequence of neighbourhoods of $(x_0,y_0)$. For each $j$ there is some $(x_j,y_j)\in R\cap W_j$ such that $\gamma_L(x_j,y_j) \neq \widetilde\gamma_L(x_j,y_j)$. This gives us a sequence $\{(x_j,y_j)\}_{j=1}^\infty$ converging to $(x_0,y_0)$ as $j\to \infty$. For each $j$ let $\omega_j$ be \[\omega_j = \pi_2\circ F^{-1}(x_j,y_j),\] namely, the initial direction of $\widetilde\gamma_L(x_j,y_j)$. This defines another sequence $\{\omega_j\}_{j=1}^\infty$. Since $S$ is compact there is a subsequence $\{\omega_{j_k}\}_{k=1}^\infty$ which converges to some $\omega\in \T S$. Now let $\widetilde\gamma_L$ be the generalised geodesic in $S_L$ generated by $(x_0,\omega)$. Then by definition $\widetilde\gamma_L$ is an $(x_0,y_0)$-geodesic in $S_L$. Let $\widetilde t_L$ be the travelling time of $\widetilde\gamma_L$, then \[\widetilde t_L = \lim_{j\to\infty}\widetilde t_L(x_j,y_j) = \widetilde t_L(x_0,y_0) = t_K.\]
 	Here we used \cref{eq:timeequality2} for the last equality. Now by \cref{eq:timeequality} we conclude that $\widetilde t_L = t_L$. But distinct $(x,y)$-geodesics have unique travelling times for all $(x,y)\in R$, and since $\gamma_L$ and $\widetilde\gamma_L$ are both $(x_0,y_0)$-geodesics with the same travelling time, they must be the same geodesic. But then applying \Cref{thm:xydependance} to $\widetilde\gamma_L$ implies that for large enough $j$, we have $\widetilde\gamma_L(x_j,y_j)=\gamma_L(x_j,y_j)$, by uniqueness of the maps $x_i^{(L)}$. This is a contradiction to our choice of sequence $\{(x_j,y_j)\}$, so we get $\widetilde\gamma_L(x,y) = \gamma_L(x,y)$ for all $(x,y)\in R\cap W$. Therefore we know that \[t_L(x,y) = \widetilde t_L(x,y) = t_K(x,y),\] for all $(x,y)\in R\cap W$. Now $t_L(x,y)$ and $t_K(x,y)$ are continuous maps which agree on the set $R\cap W$, which is dense in $W$, so they must agree on the whole of $W$.
 	
 	We have therefore shown that in a neighbourhood of $(x_0,y_0)$, the travelling times of the geodesics $\gamma_K(x,y)$ and $\gamma_L(x,y)$ agree. By \Cref{thm:timeangle}, this must mean that the incoming and outgoing directions of $\gamma_L$ and $\gamma_K$ are the same. Therefore \[\genflow^{(K)}_{t_K}(x_0, \omega_0) = \genflow^{(L)}_{t_L}(x_0,\omega_0).\]
 	So as $(x_0,\omega_0)$ tends to $(\widetilde x_0,\widetilde\omega_0)$, and therefore $t_K = t_L$ tends towards $t_0$ we get
 	\[\genflow^{(K)}_{t_0}(\widetilde x_0,\widetilde \omega_0) = \genflow^{(L)}_{t_0}(\widetilde x_0,\widetilde \omega_0),\]
 	for all $(\widetilde x_0,\widetilde\omega_0)\in \Omega$ as promised.
\end{proof}

The key idea of the following proof is that generalised geodesics generated by points in the phase space $\oT S_K\backslash Trap(S_K)$ have locally bounded travelling times. The other thing of note is that by \Cref{thm:flowequivalency}, the generalised geodesic flows of the two obstacles are identical everywhere except the interior of $S$. Using these two ideas allows us to define the homeomorphism $\Phi$.
\begin{proof}[Proof of \Cref{thm:conjugacy}]
	First, suppose $K$ and $L$ have almost the same travelling times. Let $\Omega_K = \oT{S_K}\backslash Trap(S_K)$, and define $\Omega_L$ similarly. For every point $(x,\omega)\in\Omega_K$ there is a $t(x,\omega)>0$ such that $\genflow^{(K)}_{t(x,\omega)}\in \partial S$.
	
	We claim that for each point there is a neighbourhood $U_{(x,\omega)}$ on which $t(x',\omega')$ is bounded for all $(x',\omega')\in U_{(x,\omega)}$. Suppose $(x,\omega)\in\Omega_K$ does not have such a neighbourhood. Then we can pick a converging sequence $\{(x_n,\omega_n)\}_{i=1}^\infty\to (x,\omega)$ in $\Omega_K$, such that $t(x_n,\omega_n)>t(x_{n-1},\omega_{n-1})$ and there is no $A\in\reals$ which satisfies
	\begin{equation}\label{eq:timeunbound}
		t(x_n,\omega_n)<A\text{ for all } n=1,2,\dots
	\end{equation}
	Let $t=t(x,\omega)$, then since $\genflow^{(K)}_{t}$ is continuous we have \[\{\genflow^{(K)}_{t}(x_n,\omega_n)\}_{i=1}^\infty\to \genflow^{(K)}_{t}(x,\omega).\]
	Now pick $\varepsilon>0$ small enough so that $B_{\varepsilon}(\genflow^{(K)}_{t}(x,\omega))\subseteq \uT M$ doesn't intersect $\uT K$. Then there is some $N\in\mathbb{Z}^+$ such that for all $n>N$ we have
	\[\genflow^{(K)}_{t}(x_n,\omega_n)\in B_{\varepsilon}(\genflow^{(K)}_{t}(x,\omega)).\]
	But then $t(x_n,\omega_n)<\varepsilon+t$ for all $n>N$, contradicting \cref{eq:timeunbound}. This proves our claim.
	
	We can now define the desired homeomorphism $\Phi:\Omega_K\to\Omega_L$ locally, as follows. For each $(x_0,\omega_0)\in\Omega_K$ there is a neighbourhood $U$ on which the travelling time is bounded. Let $t$ be that bound. Then define
	\begin{equation}
		\Phi|_{U}(x,\omega) = \genflow^{(L)}_{-t}\circ \genflow^{(K)}_{t}(x,\omega)
	\end{equation}
	Note that even $U$ and $V$ are two neighbourhoods of $(x_0,\omega_0)$ with different bounds $t \neq t'$, the map $\Phi$ is still well defined. Since the travelling time to the boundary is unique for each $(x,\omega)$ and by \Cref{thm:flowequivalency} the flows are identical beyond the boundary. Therefore $\Phi$ is continuous, and its inverse is defined analogously, making it a homeomorphism.
	
	Indeed, by definition, $\Phi$ is the identity map on the boundary of $S$. We show that $\Phi$ satisfies \cref{eq:conjugacy}. Given $t_0\in\reals$ and $(x,\omega)\in\Omega_K$ we have some upper bound $t$ and a neighbourhood $U$ for this bound. Then
	\begin{equation}
		\genflow^{(L)}_{t_0} \circ \Phi(x,\omega) = \genflow^{(L)}_{t_0}\circ \genflow^{(L)}_{-t}\circ \genflow^{(K)}_{t}(x,\omega)
	\end{equation}
	We use the fact that $\genflow^{(L)}$ is a flow to give
	\begin{equation}\label{eq:conj1}
		\genflow^{(L)}_{t_0} \circ \Phi(x,\omega) = \genflow^{(L)}_{-t}\circ \genflow^{(L)}_{t_0}\circ \genflow^{(K)}_{t}(x,\omega)
	\end{equation}
	Note that $\genflow^{(K)}_{t}(x,\omega)\in \overline{M\backslash S}$, and by \Cref{thm:flowequivalency} the two flows are equivalent there. So
	\begin{equation}\label{eq:conj2}
		\genflow^{(L)}_{t_0}\circ \genflow^{(K)}_{t}(x,\omega) = \genflow^{(K)}_{t_0}\circ \genflow^{(K)}_{t}(x,\omega) = \genflow^{(K)}_{t}\circ \genflow^{(K)}_{t_0}(x,\omega)
	\end{equation}
	Then combining \cref{eq:conj1} and \cref{eq:conj2} we get
	\begin{equation}
		\genflow^{(L)}_{t_0} \circ \Phi(x,\omega) = \genflow^{(L)}_{-t}\circ\genflow^{(K)}_{t}\circ \genflow^{(K)}_{t_0}(x,\omega) = \Phi\circ \genflow^{(K)}_{t_0}(x,\omega)
	\end{equation}
	Note that $\Phi$ maps non-trapped points in $\oT{S_K}$ to non-trapped points in $\oT{S_L}$. This follows from \cref{eq:conjugacy}, as follows. Given $(x,\omega)\in \Omega_K$ there is a $t_0<0$ such that $\genflow^{(K)}_{t_0}(x,\omega)\in \partial S$. Then since $\Phi$ is the identity on the boundary, by \cref{eq:conjugacy} we have \[\genflow^{(L)}_{t_0}\circ\Phi(x,\omega) = \genflow^{(K)}_{t_0}(x,\omega)\in\partial S.\]
	So $\Phi(x,\omega)$ is not trapped in $\oT{S_L}$. Therefore $K$ and $L$ have conjugate flows.
	
	Now to prove the implication in the other direction, suppose that $K$ and $L$ have conjugate flows. We show that they have almost the same travelling times. Let  $\mathcal T_K$ and $\mathcal T_L$ be the sets of travelling times in $S_K$ and $S_L$ respectively. Given any $(x,y,t)\in\mathcal T_K$, there exists an $(x,y)$-geodesic $\gamma_K$ in $S_K$ with travelling time $t$. Then $\gamma_K$ is generated by some $(x,\omega)\in \T{S}\cap \Omega_K$.
	Now, since $\Phi|_{\T{S}} = \text{Id}$ we have
	\begin{equation}
		\genflow^{(L)}_{t_0}(x,\omega) = \genflow^{(K)}_{t_0}(x,\omega)
	\end{equation}
	That is, the geodesic $\gamma_L$ in $S_L$ generated by $(x,\omega)$ has travelling time $t_0$ and is an $(x,y)$-geodesic. So we have $\mathcal T_K \subseteq \mathcal T_L$. The symmetric argument will show that $\mathcal T_L \subseteq \mathcal T_K$. Therefore $K$ and $L$ have the same travelling times.
\end{proof}

\section{Negative Curvature}\label{sec:negcurve}
The following four results have long technical proofs. For submanifolds of $\reals^n$ ($n \geq 2$) the following lemma was proved in \cite{MR1795959} (see Lemma 5.2 there). However in the case of Riemannian 2-manifolds considered here the proof is more complicated. It should be stressed that this proposition is of fundamental importance for the proof of \Cref{thm:uniqueness}. For dispersing billiards in Euclidean spaces \Cref{thm:conprop,thm:genconprop} have been well-known and widely used for a very long time - see the seminal paper of Sinai \cite{MR0274721}. It appears these propositions are known (as folklore) for the case of Riemannian 2-manifolds with negative sectional curvature as well. However we have not been able to find proofs in the literature, so we prove them in the next section for completeness.
\begin{proposition}
\label{thm:convexlemma}
	Suppose $c:[a,b]\to M$ is a smooth, unit speed, strictly convex curve. For each $u_0\in [a,b]$ there exists  a smooth, strictly convex curve $y$ on a neighbourhood of $u_0$ such that $\pdiff{u}{y}(u)$ is orthogonal to the parallel translate of $\pdiff{u}{c}(u)$. 
\end{proposition}

\begin{proposition}\label{thm:conprop}
	Suppose $M$ has non-positive sectional curvature. Given a strictly convex, unit speed curve $x:[a,b]\to M$, let $\omega:[a,b]\to\uT{M}$ be the unit outward normal field to it. Define the map $Y:[a,b]\times\reals\to M$ by
	\begin{equation}\label{eq:propfamily}
		Y(u,t) = \flow_{t}(x(u),\omega(u))
	\end{equation}
	For each $t_0\in\reals$ the curve $y(u)=\pi_1\circ Y(u,t_0)$ is strictly convex.
\end{proposition}

\begin{proposition}\label{thm:genconprop}
	Suppose $M$ has non-positive sectional curvature, and that $K$ is a disjoint union of strictly convex obstacles. Then convex fronts in $S_K$ which hit an obstacle transversely (i.e. not tangentially) will remain convex after reflection.
\end{proposition}

\begin{lemma}\label{thm:nontrappedboundary}
	In every open set $U$ of $\uT\partial K$ there are infinitely many points which generate a non-trapped and non-tangent geodesic in $S_K$. In fact, the set of these points has topological dimension 1, and its complement has dimension 0.
\end{lemma}
\begin{proof}
	See \Cref{appendix}
\end{proof}

\begin{figure}
	\center
	\def\svgwidth{0.9\linewidth}
	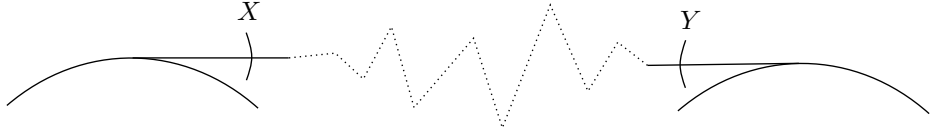
	\caption{Two convex fronts on a collision course}\label{fig:convexcollision}
\end{figure}

We now prove two useful results and then prove \Cref{thm:uniqueness}.

\begin{lemma}\label{thm:confrontorth}
	Let $\gamma$ be a generalised geodesic in $S_K$. Suppose there are two convex fronts, $X$ and $Y$ such that $\dot\gamma(0)$ is the outward unit normal to $X$ and for some $t_0>0$, the velocity $\dot\gamma(t_0)$ is the inward unit normal to $Y$ (\cref{fig:convexcollision}). Also Suppose that $\gamma$ reflects transversally between $X$ and $Y$. Parameterise $X$ as \[x:[a,b]\to S_K\] with $x(u_0)=\gamma(0)$ and unit outward normal $\omega(u)$. Then there exists an open set $U\subseteq [a,b]$ such that $(x(u_0),\omega(u_0))$ is the only point in the normal bundle of $X$ which generates a geodesic that hits $Y$ orthogonally.
\end{lemma}
\begin{proof} First, suppose that $\gamma$ hits $Y$ before any reflections.
	Let \[y(u,t)=\genflow_t(x(u),\omega(u)),\] then possibly after shrinking $[a,b]$ around $u_0$, there is a function $t:[a,b]\to\reals$ such that \[y(u) = y(u,t(u))\subseteq Y\] and  $\gamma(t_0) = y(u_0)$. Pick Riemannian normal coordinates about $y(u_0)$. By \Cref{thm:conprop} we may assume that $X$ is within the chart (shrinking and propagating it forward until it is). We can then give the Taylor expansion of $y(u,t)$ locally. It is in fact the same as \cref{eq:taylory} from \Cref{thm:genconprop}. The first and second derivative of $y(u)$ are also exactly the same as \cref{eq:yder} and \cref{eq:yder2} respectively.
	Note that in this chart the second derivative of $y$ and the covariant derivative of $\pdiff{u}y$ are equivalent at $y(u_0)$. We then define the function
	\begin{equation}
		f(u) = \left<\Pdiff{y}{u}(u),\omega(u)\right> = \Pdiff{t}{u}(u) + \O(t^2) + \O(tg)
	\end{equation}
	Note that $f(u_0) = 0$. Also note that taking inner products with $\omega$ and \cref{eq:yder2} will give the following identity,
	\begin{equation}
		\Pdifftwo{t}{u}(u_0) = -K_Y - K_X + t\norm{\Pdiff{\omega}{u}(u_0)}^2 - \O(t) - \O(g)
	\end{equation}
	Here $K_X<0$ and $K_Y<0$ are the curvatures of $X$ and $Y$ at $x(u_0)$ and $y(u_0)$ respectively.
	We evaluate the derivative of $f$ at $u_0$:
	\begin{align}
		\Pdiff{f}{u}(u_0) &= \Pdifftwo{t}{u}(u_0) + \O(t) + \O(g)\\
		&= - K_Y - K_X + t\norm{\Pdiff{\omega}{u}(u_0)}^2 + \O(t) + \O(g)
	\end{align}
	Note that the last two terms can be made as small as we like, by bringing $x(u)$ forward. Therefore the derivative of $f$ at $u_0$ is positive. So we can shrink $[a,b]$ around $u_0$ such that $f$ has positive derivative on all of $[a,b]$. Now let $Z = f^{-1}(0)$. Then $Z$ is a closed 0-dimensional submanifold of $[a,b]$. Since $Z$ is closed and discrete, and $[a,b]$ is compact, $Z$ must be finite. Then since $[a,b]$ is Hausdorff we can find an open set $U\subseteq [a,b]$ such that $U\cap Z = \{u_0\}$. Note that a point $(x(u),\omega(u))$ in the normal bundle of $X$ generates a geodesic that hits $Y$ orthogonally if and only if $u\in Z$.
	
	Now we will deal with the case of reflections. Suppose that $X$ is one reflection away from $Y$. We wish to show that you may shrink $[a,b]$ around $u_0$ so that $X$ will hit $\partial K$ transversally. Then we may apply \Cref{thm:genconprop} and the result will follow from the argument above. Suppose $v\in [a,b]$ is such that $y(v,t_v)$ will hit $\partial K$ tangentially for some $t_v>0$. Then we may parameterise $\partial K$ in a closed neighbourhood $U_v$ of $y(v,t_v)$, as \[k:[-\delta,\delta]\to U_v\text{ such that }k(0) = y(v,t_v).\]
	Then by \Cref{thm:convexlemma} we can construct a convex front $\widetilde X$ such that the tangents of $U_v$ will hit $\widetilde X$ orthogonally. Now by the argument above, there are only finitely many $u\in [a,b]$ such that $y(u,t)$ will hit $\widetilde X$ orthogonally. Thus there are only finitely many $u\in [a,b]$ which will hit $\partial K$ tangentially (before reflecting). Thus we may shrink $[a,b]$ around $u_0$ so that $X$ will hit $\partial K$ transversally before reaching $Y$.
	
	The result now follows for any finite number of reflections.
\end{proof}

\begin{corollary}\label{thm:onetangent}
	The set of points in $\uT \partial K$ which generate a generalised geodesic that is tangent to any obstacle is countable.
\end{corollary}
\begin{proof}
	Since we can pick a countable subcover of $\uT \partial K$ it suffices to show that this holds in every neighbourhood $U\subseteq \uT\partial K$. We construct the convex front $\Sigma$ from $U$ as in \Cref{thm:convexlemma}, parameterised on some $(a,b)$.
	Let $R_i$ be the set of points in $(a,b)$ which generate a geodesic that is tangent after $i$ reflections. By \Cref{thm:confrontorth}, the set $R_1$ is finite. Now suppose that $R_i$ is countable for all $i<n$. Then \begin{equation}
		\widetilde R_{n-1} = \bigcup^{n-1}_{i=1} R_i
	\end{equation}
	is a countable set. We can pick some $u_0\in\widetilde R_{n-1}$ and order the rest such that
	\begin{equation}
		a<\cdots<u_{-1}<u_0<u_1<\cdots<b
	\end{equation}
	Then $\Sigma$ restricted to each $(u_i,u_{i+1})$ we will have transversal reflections for the first $n-1$ times. Now by the same argument as \Cref{thm:confrontorth}, there will be at most countably many tangencies after $n-1$ reflections in $(u_i,u_{i+1})$. For the edge cases, there is either a minimum $u_a$ such that $a<u_a<u_i$ for all $i$. In which case $(a,u_a)$ will also have at most countably many tangencies. Or if no such $u_a$ exists then we must have
	\begin{equation}
		(a,u_0) = \bigcup_{i=0}^\infty (u_{-i-1},u_{-i})
	\end{equation}
	A similar argument will follow for the upper bound $b$ as well. Thus in all cases the set
	\begin{equation}
		R'_n = \bigcup_{i=-\infty}^\infty (u_{i-1},u_i)
	\end{equation}
	Will be countable, (possibly including $(a,u_a)$ and $(b,u_b)$ if necessary). Thus $R_n \subseteq R'_n\cup\widetilde R_{n-1}$ so $R_n$ must be countable as well. It follows by induction that the set $R$ of tangencies in $(a,b)$, must be countable as well, since
	\begin{equation}
		R = \bigcup_{i=1}^\infty R_n
	\end{equation}
\end{proof}

In the following proof we assume $K\neq L$ and construct a convex front from an obstacle in $S_K$ and one in $S_L$ such that they contradict \Cref{thm:confrontorth}.
\begin{proof}[Proof of \Cref{thm:uniqueness}]
	First, we show that $K\subseteq L$. Suppose otherwise, and pick some point $x_0\in\partial K$ such that $x_0\not\in L$. Then there is a small neighbourhood $U\subseteq\partial K$ of $x_0$ such that $U\cap L=\emptyset$. By \Cref{thm:convexlemma} and possibly shrinking $U$ we can construct a convex front $X$ such that the geodesics generated by the tangents of $U$ will hit $X$ orthogonally. Parameterise $X$ around $x_0$ on a closed interval, so that the image of the parametrisation is a compact convex front. We shrink $X$ to that image and treat $X$ as compact from now on. Let $\widetilde X$ be the set of points $(x,v_x)$ where $x\in X$ and $v_x$ is the outward unit normal to $X$ at $x$ such that:
	 \begin{enumerate}
	 	\item  $(x,v_x)$ generates a non-trapped geodesic in $S_K$ and $S_L$
	 	\item  $(x,v_x)$ generates a geodesic with exactly one point of tangency in $S_K$ and exactly on point of tangency in $S_L$
	 \end{enumerate}
	 Then by \Cref{thm:onetangent} and \Cref{thm:nontrappedboundary}, the set $\widetilde X$ is infinite (in fact, it has topological dimension 1). Note that demanding a tangency for both $S_K$ and $S_L$ is not difficult, since by \Cref{thm:conjugacy}, $K$ and $L$ have conjugate flows. Meaning they have the same travelling times (by \Cref{thm:conjugacy}). There will be a singularity in the travelling time function at the point of tangency for $K$ and therefore for $L$ too since the two functions are equal. Thus a tangency in $S_K$ will guarantee a tangency in $S_L$
	
	Given $(x,v_x)\in \widetilde X$, the geodesic $\gamma_K$ in $S_K$ generated by $(x,v_x)$ will have exactly one point of tangency by definition.  So will $\gamma_L$, the corresponding geodesic in $S_L$, have exactly one point of tangency. However, since $\gamma_K$ is tangent to $\partial K$ at some point in $U$, we know $\gamma_L$ must be tangent at some other point $z(x,v_x)\in \partial L$. Let 
	\begin{equation}
		Z = \{z(x,v_x):(x,v_x)\in \widetilde X\}
	\end{equation}
	Then, by uniqueness of geodesics, $z:\widetilde X\to Z$ is a bijection.
	
	Suppose for a moment that for every $z\in\partial L$ there is a neighbourhood $W_z\subseteq\partial L$ such that $W_z\cap Z$ is finite. Then picking a finite cover $W_1,\dots,W_n$ of these we see that $Z = \cup_{i=1}^n W_i\cap Z$ is finite, a contradiction. Thus there is some $z_0\in\partial L$ such that for any neighbourhood $W\subseteq \partial L$ of $z_0$, the set $\widetilde W = W\cap Z$ is infinite.
	
	Now since $\widetilde W$ is dense in $W$, we can pick a sequence $\{z_i\}_{i=1}^\infty \to z_0$ of points $z_i\in\widetilde W$. Then for each $z_i$ there is a corresponding $(x_i,v_i)\in\widetilde X$, such that \[\pi_1\circ \genflow^{(L)}_{t_i}(x_i,v_i)=z_i,\] for some time $t_i\in\reals$. If $t_i<0$ we may construct $X$ in the opposite direction so that $t_i>0$. Therefore we assume $t_i>0$. Construct $Y$ in $S_L$ via \Cref{thm:convexlemma} from $W$ around $z_0$, such that the outward unit normals to $Y$ will point toward $X$.
	
	We can shrink $X$ so that there will be infinitely many $(x_i, v_i)$ in $X$, and it will propagate transversely until it hits $Y$. Since $X$ is compact, we may take a convergent subsequence of $\{(x_i,v_i)\}_{i=1}^\infty$, and denote the point it converges to as $(\widetilde x,\widetilde v)$. Note that no matter how much we shrink $X$ around $(\widetilde x, \widetilde v)$ we must have infinitely many $(x_i,v_i)\in X$ since the subsequence we took converges to $(\widetilde x, \widetilde v)$. Thus we can safely shrink $X$ around $(\widetilde x,\widetilde v)$ so that it propagates transversally until hitting $Y$.
	
	Then by our construction, every $(x_i,v_i)$ in $X$ will generate a geodesic which hits $Y$ orthogonally. There are infinitely many such $(x_i,v_i)$ in $X$, a contradiction to \Cref{thm:confrontorth} which asserts that there must only be finitely many of these. So $K\subseteq L$, and by applying the same argument, $K = L$. 
\end{proof}

\section{Proofs of Propositions 9, 10 and 11}
\label{sec:hardproofs}

The common difficulty shared by the proofs in this section is exhibiting the curvature of the convex front, whether it be from the influence of the manifolds intrinsic negative sectional curvature as in \Cref{thm:conprop,thm:genconprop}, or simply by construction as in \Cref{thm:convexlemma}. Although \Cref{thm:conprop} is done in a completely coordinate-free way, in both \Cref{thm:convexlemma,thm:genconprop} we use a Riemannian coordinate chart to allow us to use a Taylor expansion of the geodesics in question.

\begin{proof}[Proof of \Cref{thm:convexlemma}]
	Define the family of geodesics $y:[a,b]\times[0,1]\to M$ as
	\begin{equation}
		y(u,t) = \pi_1\circ\flow_t\left(c(u),\Pdiff{c}{u}(u)\right)
	\end{equation}
	Note that since $y$ is a variation through geodesics, $\Pdiff{y}{u}$ is a Jacobi field at each $u\in [a,b]$ which satisfies the initial conditions:
	\begin{equation} 
		\Pdiff{y}{u}(u,0) = \Pdiff{c}{u}(u)
	\end{equation}
	\begin{equation}
		\cdiff{t}\Pdiff{y}{u}(u,0) = \cdiff{u}\Pdiff{c}{u}(u)
	\end{equation}
	Furthermore, since $\Pdiff{y}{u}$ is a Jacobi field, it satisfies the differential equation:
	\begin{equation}\label{eq:jacobifield}
		\cdifftwo{t}\Pdiff{y}{u} + R\left(\Pdiff{y}{u},\Pdiff{y}{t}\right)\Pdiff{y}{t} = 0
	\end{equation}
	In this construction, the parallel translate of $\Pdiff{c}{u}$ is the tangent vector $\Pdiff{y}{t}$, so we wish to construct a curve $\widetilde y$  such that $\Pdiff{\widetilde y}{u}$ is orthogonal to $\Pdiff{y}{t}$. Consider a curve of the form $\widetilde y(u) = y(u,r(u))$,
	where $r:[a,b]\to\reals$ is some smooth function. We want to enforce the following condition,
	\begin{equation}\label{eq:orthcondition}
		\left<\Pdiff{\widetilde y}{u},\Pdiff{y}{t}\right> = 0
	\end{equation}
	Noting that
	\begin{equation}
		\Pdiff{\widetilde y}{u} = \Pdiff{y}{u} + \Pdiff{r}{u}\Pdiff{y}{t}
	\end{equation}
	We can see that \cref{eq:orthcondition} is equivalent to
	\begin{equation}
		\Pdiff{r}{u} = -\left<\Pdiff{y}{u},\Pdiff{y}{t}\right>
	\end{equation}
	So consider the function $g:[a,b]\times [0,1]\to\reals$ given by
	\begin{equation}
		g(u,t) = \left<\Pdiff{y}{u}(u,t),\Pdiff{y}{t},(u,t)\right>
	\end{equation}
	Then taking derivatives with respect to $t$ we get,
	\begin{equation}
		\Pdiff{g}{t} = \left<\cdiff{t}\Pdiff{y}{u},\Pdiff{y}{t}\right> + \left<\Pdiff{y}{u},\cdiff{t}\Pdiff{y}{t}\right> = \left<\cdiff{t}\Pdiff{y}{u},\Pdiff{y}{t}\right>
	\end{equation}
	Here, the second term evaluates to zero since $y$ is a geodesic in the $t$ direction. Now taking derivatives once more yields
	\begin{equation}
		\Pdifftwo{g}{t} = \left<\cdifftwo{t}\Pdiff{y}{u},\Pdiff{y}{t}\right> = -\left< R\left(\Pdiff{y}{u},\Pdiff{y}{t}\right)\Pdiff{y}{t}, \Pdiff{y}{t}\right>
	\end{equation}
	Here we used \cref{eq:jacobifield} to get the second equality. Note that there are three instances of $\Pdiff{y}{t}$ in the last inner product, so the anti-symmetry properties of the Riemannian curvature ensure that this inner product evaluates to zero. That is, $g$ must have the form $g(u,t) = \alpha(u) + \beta(u)t$, for some smooth functions $\alpha,\beta:[a,b]\to\reals$. Note that $g(u,0) = \alpha(u)$. So we can directly evaluate
	\begin{equation}
		\alpha(u) = \left<\Pdiff{y}{u}(u,0),\Pdiff{y}{t},(u,0)\right> = \norm{\Pdiff{c}{u}(u)}^2 = 1
	\end{equation}
	We can also find $\beta(u)$ as follows, using the initial conditions for $y$,
	\begin{equation}
		\beta(u) = \Pdiff{g}{t}(u,0) = \left<\cdiff{t}\Pdiff{y}{u}(u,0),\Pdiff{y}{t}(u,0)\right> = 0
	\end{equation}
	Therefore \cref{eq:orthcondition} reduces to
	\begin{equation}
		\Pdiff{r}{u}(u) = -1
	\end{equation}
	So we can define our convex curve as $\widetilde y(u) = y(u,\lambda-u)$, where $\lambda\in\reals$ is some small constant. Now all that is left to show is the strict convexity of $\widetilde y$. To do so, fix $u_0\in[a,b]$. We show that $\widetilde y$ is strictly convex in a neighbourhood of $u_0$. That is, for some $\lambda\in\reals$ we get
	\begin{equation}\label{eq:convexitycondition}
		\left<\cdiff{u}\Pdiff{\widetilde y}{u}(u_0), \Pdiff{y}{t}(u_0,\lambda-u_0) \right> < 0
	\end{equation}
	Consider the following function:
	\begin{equation}
		f(u,t) = \left<\cdiff{u}(\Pdiff{y}{u}(u,t)-\Pdiff{y}{t}(u,t)), \Pdiff{y}{t}(u,t)\right>
	\end{equation}
	The condition \cref{eq:convexitycondition} is equivalent to $f(u_0,\lambda-u_0)<0$. Note that
	\begin{equation}
		\left<\Pdiff{y}{u}-\Pdiff{y}{t},\Pdiff{y}{t}\right> = 0.
	\end{equation} 
	So we can re-write $f$ as follows:
	\begin{align}
		f = \pdiff{u}\left<\Pdiff{y}{u}-\Pdiff{y}{t},\Pdiff{y}{t}\right> - \left<\Pdiff{y}{u}-\Pdiff{y}{t},\cdiff{u}\Pdiff{y}{t}\right> = - \left<\Pdiff{y}{u}-\Pdiff{y}{t},\cdiff{t}\Pdiff{y}{u}\right>
	\end{align}
	In the last step we also used the symmetry
	\begin{equation}
		\cdiff{t}\Pdiff{y}{u} = \cdiff{u	}\Pdiff{y}{t}
	\end{equation}
	Now consider the derivative of $f$ along $t$,
	\begin{align}
		\Pdiff{f}{t} &=  - \left<\cdiff{t}(\Pdiff{y}{u}-\Pdiff{y}{t}),\cdiff{t}\Pdiff{y}{u}\right> -  \left<\Pdiff{y}{u}-\Pdiff{y}{t},\cdifftwo{t}\Pdiff{y}{u}\right>\\
		&= -\left<\cdiff{t}\Pdiff{y}{u},\cdiff{t}\Pdiff{y}{u}\right>+\left<\Pdiff{y}{u}-\Pdiff{y}{t}, R\left(\Pdiff{y}{u},\Pdiff{y}{t}\right)\Pdiff{y}{t}\right>\\
		&= -\norm{\cdiff{u}\Pdiff{y}{t}}^2 + \left<\Pdiff{y}{u}, R\left(\Pdiff{y}{u},\Pdiff{y}{t}\right)\Pdiff{y}{t}\right>\label{eq:curvaturederivative}
	\end{align}
	In the second line we use the fact that $y$ is a geodesic in the $t$ direction and \cref{eq:jacobifield} to simplify the expression. Then in the third line we used the symmetry of the covariant and partial derivatives again, and the anti-symmetry properties of the Riemannian curvature as before. Now taking some coordinates around $c(u_0)$ and ensuring that $\lambda$ is small enough so that $\widetilde y(u_0)$ is within the chart, we can estimate $f(u_0,\lambda-u_0)$ using a Taylor expansion:
	\begin{equation}
		f(u_0,\lambda-u_0) = f(u_0,0) + (\lambda-u_0) \Pdiff{f}{t}(u_0,0) + \O(\lambda-u_0)^2
	\end{equation}
	Note that we can directly evaluate $f(u_0,0) = 0$. We can also give Taylor expansions for the derivatives of $y$:
	\begin{equation}
		\Pdiff{y}{t}(u,t) = \Pdiff{c}{u}(u) + O(t)
	\end{equation}
	\begin{equation}
		\Pdiff{y}{u}(u,t) = \Pdiff{c}{u}(u) + O(t)
	\end{equation}
	Therefore the curvature term in \cref{eq:curvaturederivative} will be zero up to a first order approximation. That is,
	\begin{equation}
		\Pdiff{f}{t}(u_0,0) = -\norm{\cdiff{u}\Pdiff{c}{u}(u_0)}^2 + O(\lambda-u_0)^2
	\end{equation}
	Note that the convexity of the curve $c$ ensures that the first term is nonzero. So substituting back into the Taylor expansion we get
	\begin{equation}
		f(u_0,\lambda-u_0) = -(\lambda-u_0) \norm{\cdiff{u}\Pdiff{c}{u}(u_0)}^2 + \O(\lambda-u_0)^2
	\end{equation}
	Thus we can pick $\lambda>u_0$ such that $\lambda - u_0$ is small enough to ensure the curve $\widetilde y$ is strictly convex at $u_0$. Then by continuity, $\widetilde y$ is strictly convex in a neighbourhood of $u_0$ as desired.
\end{proof}

\begin{proof}[Proof of \Cref{thm:conprop}]
	Given a convex, unit speed curve $x:[a,b]\to M$ and its unit outward normal $\omega(u)$. Let
	\begin{equation}
		y(u,t) = \pi_1\circ \flow_{t}(x(u),\omega(u))
	\end{equation}
	Define the map $f:[a,b]\times\reals\to\reals$ by
	\begin{equation}
		f(u,t) = \left< \cdiff{u}\Pdiff{y}{u}(u,t), \Pdiff{y}{t}(u,t) \right>
	\end{equation}
	This map is the convexity of the propagated curve at each point $y(u,t)$. Since $x$ is a convex curve we have
	\begin{equation}
		f(u,0) = \left< \cdiff{u}\Pdiff{x}{u}(u), \omega(u) \right> < 0
	\end{equation}
	We investigate how the convexity changes over time as the curve is propagated forward. Consider the following derivative:
	\begin{equation}
		\pdiff{t}f =  \left<\cdiff{t} \cdiff{u}\Pdiff{y}{u}, \Pdiff{y}{t} \right> +  \left< \cdiff{u}\Pdiff{y}{u}, \cdiff{t} \Pdiff{y}{t} \right>
	\end{equation}
	Note that keeping $u$ fixed, $y(u,t)$ is a geodesic, so 
	\[\cdiff{t}\pdiff{t}y(u,t) = 0\]
	We now use the symmetry 
	\begin{equation}
		\cdiff{t} \cdiff{u}\Pdiff{y}{u} - \cdiff{u} \cdiff{t}\Pdiff{y}{u} = R\left(\Pdiff{y}{u},\Pdiff{y}{t}\right)\Pdiff{y}{u}
	\end{equation}
	Then we can write
	\begin{equation}\label{eq:convexityderivative}
		\pdiff{t}f = \left<\cdiff{u} \cdiff{t}\Pdiff{y}{u}, \Pdiff{y}{t} \right> + \left<R\left(\Pdiff{y}{u},\Pdiff{y}{t}\right)\Pdiff{y}{u},\Pdiff{y}{t}\right>
	\end{equation}
	Now we focus on the first term on the right hand side:
	\[\left<\cdiff{u} \cdiff{t}\Pdiff{y}{u}, \Pdiff{y}{t} \right> = \left<\cdifftwo{u}\Pdiff{y}{t}, \Pdiff{y}{t} \right>\]
	Since $\Pdiff{y}{t}$ is the parallel translate of $\omega$, it has unit norm. Therefore we can write
	\begin{equation}\label{eq:conder1}
		\left<\cdifftwo{u}\Pdiff{y}{t}, \Pdiff{y}{t} \right> = -\norm{\cdiff{t}\Pdiff{y}{t}}
	\end{equation}
	Now for the second term in \cref{eq:convexityderivative}, we simply note that $\pdiff{u}y$ and $\pdiff{t}y$ are orthonormal, so we have the sectional curvature
	\begin{equation}\label{eq:conder2}
		\left<R\left(\Pdiff{y}{u},\Pdiff{y}{t}\right)\Pdiff{y}{u},\Pdiff{y}{t}\right> = K\left(\Pdiff{y}{u},\Pdiff{y}{t}\right)\leq 0
	\end{equation}
	Finally putting \cref{eq:conder1} and \cref{eq:conder2} together with \cref{eq:convexityderivative} we get
	\begin{equation}
		\pdiff{t}f = -\norm{\cdiff{t}\Pdiff{y}{t}} + K(\Pdiff{y}{u},\Pdiff{y}{t}) < 0
	\end{equation}
	Therefore the curve $y(u,t)$ will become more convex the further along it is propagated.
\end{proof}
In \cite{LUCHOBOOK} (Lemma 10.1.3) there is a proof given for the Euclidean case. We generalise this argument to the Riemannian case which is considerably more difficult.

\begin{proof}[Proof of \Cref{thm:genconprop}]
	Given a convex, unit speed curve $x:[a,b]\to S_K$ and its unit outward normal $\omega(u)$, define the propagated curve
	\begin{equation}
		Y(u,t) = \genflow_t(x(u),\omega(u))
	\end{equation}
	By \Cref{thm:conprop} we know that $y(u,t) = \pi_1\circ Y(u,t)$ will remain convex for any fixed $t\in\reals$ as long as it is entirely in the interior of $S_K$. i.e. before any reflections. So all that remains to show is that $y(u,t)$ will stay convex after reflecting off an obstacle. We assumed that all of $x(u)$ will hit an obstacle $K$ transversally, so there is a smooth time function $t:[a,b]\to\reals$ such that $y(u,t(u))\in \partial K$ for all $u\in [a,b]$. Define
	\begin{equation}
		z(u,t) = y(u,t(u) + t) = \pi_1\circ\genflow_{t}(Y(u,t(u)))
	\end{equation}
	 Let $T>t(u_0)$ be some time after the reflection but before any more reflections with another obstacle. Define $\varepsilon = T-t(u)$ and let $z(u) = z(u,\varepsilon)$. Also let $y(u) = y(u,t(u))$, and define $v(u)$ to be the unit outward normal to $\partial K$ at $y(u)$.
	We wish to consider how the curvature of $y(u,t)$ changes at the point of reflection $y(u)$.
	Fix $u_0\in [a,b]$ and consider some Riemannian coordinates about the point $y_0 = y(u_0)$. To ensure that $x(u)$ and $z(u)$ are both contained in the chart we may have to decrease $\varepsilon$ somewhat and replace $x(u)$ with $y(u,\varepsilon')$ for some $\varepsilon'>0$.
	We give Taylor expansions of both $y(u,t)$ and $z(u,t)$ locally.
	\begin{align}
		y(u,t) &= y(u,0) + t\Pdiff{y}{t}(u,0) + \frac{1}{2}t^2\Pdifftwo{y}{t}(u,0) + \cdots\\
			   &= x(u) + t\omega(u) + \frac{1}{2}t^2g(u) + \O(t^3) \label{eq:taylory}
	\end{align}
	Now we use the fact that $y$ is a geodesic along the $t$ variable to give
	\begin{equation}\label{eq:christoffel}
		g^k(u) = \Pdifftwo{y^k}{t}(u,0) = -\sum_{i,j=0}^{2}\Pdiff{y^i}{t}(u,0)\Pdiff{y^j}{t}(u,0)\Gamma_{ij}^k(x(u))
	\end{equation}
	Note that as $t(u_0)\to 0$ by moving the curve $x(u)$ forward, the Christoffel symbols will tend to zero, $\Gamma_{ij}^k(x(u_0))\to \Gamma_{ij}^k(y(u_0)) = 0$.
	So $g(u_0)$ is a term which tends to zero as $t(u_0)\to 0$.
	Now we find the Taylor expansion for $z(u,t)$ similarly
	\begin{equation}
		z(u,t) = y(u) + t\widetilde\omega(u) + \frac{1}{2}t^2\widetilde g(u) + \O(t^3)
	\end{equation}
	Here $\widetilde\omega(u)$ is the direction of the normal after reflection. That is, 
	\begin{equation}\label{eq:inversion}
		\widetilde\omega(u) = \omega(u) - 2\left<\omega(u),v(u)\right>v(u)
	\end{equation}
	And $\widetilde g$ is defined similarly to $g$ from \cref{eq:christoffel},
	\begin{equation}
		\widetilde g^k(u) = \Pdifftwo{z^k}{t}(u,0) = -\sum_{i,j=0}^{2}\Pdiff{z^i}{t}(u,0)\Pdiff{z^j}{t}(u,0)\Gamma_{ij}^k(y(u))
	\end{equation}
	Now we calculate the derivative of $y(u)$, which will be helpful later,
	\begin{equation}\label{eq:yder}
		\Pdiff{y}{u} = \Pdiff{x}{u} + \Pdiff{t}{u}\omega + t\Pdiff{\omega}{u} + \O(t^2) + \O(tg)
	\end{equation}
	And also the second derivative:
	\begin{equation}\label{eq:yder2}
		\Pdifftwo{y}{u} = \Pdifftwo{x}{u} + \Pdifftwo{t}{u}\omega + 2\Pdiff{t}{u}\Pdiff{\omega}{u} + t\Pdifftwo{\omega}{u} + \O(t) + \O(g)
	\end{equation}
	Similarly we find the first and second derivatives of $z(u)$,
	\begin{equation}\label{eq:zder}
		\Pdiff{z}{u} = \Pdiff{y}{u} - \Pdiff{t}{u}\widetilde\omega + \varepsilon\Pdiff{\widetilde\omega}{u} +\O(\varepsilon^2) + \O(\varepsilon\widetilde g)
	\end{equation}
	\begin{equation}\label{eq:zder2}
		\Pdifftwo{z}{u} = \Pdifftwo{y}{u} - \Pdifftwo{t}{u}\widetilde\omega - 2\Pdiff{t}{u}\Pdiff{\widetilde\omega}{u} + \varepsilon\Pdifftwo{\widetilde\omega}{u} +\O(\varepsilon) + \O(\widetilde g)
	\end{equation}
	Now from \cref{eq:yder}, taking inner products with $\omega$ will give the following identity,
	\begin{equation}
		\Pdiff{t}{u} = \left<\Pdiff{y}{u},\omega\right> - \O(t^2) - \O(tg)
	\end{equation}
	Combining this with \cref{eq:zder} and taking inner products with $\widetilde\omega$ we get:
	\begin{equation}
		\left<\Pdiff{z}{u},\widetilde\omega\right> = \left<\Pdiff{y}{u},\widetilde\omega-\omega\right> + \O(t^2) + \O(\varepsilon^2) + \O(tg) + \O(\varepsilon\widetilde g)
	\end{equation}
	Bringing $x(u_0)$ and $z(u_0)$ close to $y(u_0)$ we can see that $t(u_0)\to 0$ and $\varepsilon\to 0$, giving us
	\[\left<\Pdiff{z}{u},\widetilde\omega\right> = \left<\Pdiff{y}{u},\widetilde\omega-\omega\right> = 0.\]
	Here the right hand side evaluates to zero by \cref{eq:inversion}. So $\widetilde\omega(u_0)$ is indeed the outward unit normal to $z(u_0)$. Therefore $z(u)$ will be convex at $u_0$ if
	\begin{equation}\label{eq:curvaturez}
		\left<\cdiff{u}\Pdiff{z}{u}(u_0),\widetilde\omega(u_0)\right> < 0
	\end{equation}
	Note that letting $t(u_0)\to 0$ and $\varepsilon\to 0$, the covariant derivative of $z$ will be almost exactly the second derivative. That is,
	\begin{equation}
		\cdiff{u}\Pdiff{z^k}{u}(u_0) = \Pdifftwo{z}{u}(u_0) + \Pdiff{z^i}{u}(u_0)\Pdiff{z^j}{u}(u_0)\Gamma^{k}_{ij}(z(u_0))
	\end{equation}
	Here, once again, the Christoffel symbols will vanish as $t(u_0)\to 0$.
	Now using \cref{eq:yder2} and taking the inner product with $\omega$ gives
	\begin{equation}\label{eq:tder2}
		\Pdifftwo{t}{u}(u_0) = \left<\Pdifftwo{y}{u}(u_0),\omega(u_0)\right> - \left<\Pdifftwo{x}{u}(u_0),\omega(u_0)\right>
	\end{equation}
	Now taking the inner product of \cref{eq:zder2} and $\widetilde\omega$ and combining with \cref{eq:tder2} gives
	\begin{equation}\label{eq:zcurv}
		\left<\Pdifftwo{z}{u}(u_0),\widetilde\omega(u_0)\right> = \left<\Pdifftwo{y}{u}(u_0),\widetilde\omega(u_0)-\omega(u_0)\right> + \left<\Pdifftwo{x}{u}(u_0),\omega(u_0)\right>\nonumber
	\end{equation}
	Note that the second term is negative since $x(u)$ is convex. And by \cref{eq:inversion} we have
	\begin{equation}\label{eq:curvdiff}
		\left<\Pdifftwo{y}{u}(u_0),\widetilde\omega(u_0)-\omega(u_0)\right> = -2\left<\omega(u_0),v(u_0)\right>\left<\Pdifftwo{y}{u}(u_0),v(u_0)\right>
	\end{equation}
	Since $\omega(u_0)$ points into $\partial K$ and $v(u_0)$ point outwards, their inner product will be negative. Moreover the convexity of $\partial K$ implies that the second term on the right hand side will be negative as well. So \cref{eq:curvdiff} will be negative. Thus combining with \cref{eq:zcurv} we see that \cref{eq:curvaturez} is satisfied, making $z$ convex at $u_0$.
\end{proof}

\appendix
\section{}
\label{appendix}
\begin{proof}[Proof of \Cref{thm:distincttimes}]\label{proof:distincttimes}
	Suppose $\omega_0$ and $\rho_0$ are distinct outgoing directions at $x_0$ such that the geodesics generated by them, $\gamma_{(x_0,\omega_0)}$ and $\gamma_{(x_0,\rho_0)}$, have the same endpoint $y_0\in \partial S$ and the same travelling time $t_0$. That is, $\gamma_{(x_0,\omega_0)}(t_0) = \gamma_{(x_0,\rho_0)}(t_0) = y_0$.
	Let $\omega:U\to \T{S}_{x_0}$ and $\rho:V\to \T{S}_{x_0}$ be any two local parametrisation of neighbourhoods $U$ and $V$ of $\omega_0$ and $\rho_0$ respectively. Also suppose that $\omega(u)\neq \rho(v)$ always. Let $t_1(u)$ and $t_2(v)$ be the travelling time functions for the two geodesics respectively, and define their difference $T(u,v) = t_1(u)-t_2(v)$. Now let $\alpha : W\to \partial S$ be a local parametrisation of $\partial S$ around $y_0$. Let $y(u)$ and $z(v)$ be the endpoints of $\gamma_{(x_0,\omega(u))}$ and $\gamma_{(x_0,\rho(v))}$ respectively. Also, let $\widetilde{u}(u)$ be such that $\alpha(\widetilde{u}) = y(u)$ and similarly with $\widetilde{v}(v)$. 
	Now we can define a function $H(u,v) = \widetilde{u}-\widetilde{v}$ and thus define a function $G(u,v) = (T(u,v),H(u,v))$. Note that by definition $G(u,v) = (0,0)$ if and only if $\omega(u)$ and $\rho(v)$ generate a geodesic with the same endpoint and travelling times. Now consider the derivative of $G$. It is only nonsingular if its columns $grad~T$ and $grad~H$ are linearly independent. We show that this is the case for all $(u,v)\in G^{-1}(0,0)$. First consider $\pdiff{u}T(u,v)$, by \Cref{thm:timeangle} we know that
	\[ \pdiff{u}T(u,v) = \left<\sigma(u),\pdiff{u} y(u)\right> = \left<\sigma(u), \alpha'(\frac{\partial\widetilde{u}}{\partial u}) \right>. \]
	Here $\sigma(u)$ is the incoming direction of the geodesic $\gamma_{(x_0,\omega(u))}$ at the endpoint. We also know that
	\[ \pdiff{u}H(u,v) = \frac{\partial\widetilde{u}}{\partial u},\] 
	and similarly with respect to $v$. Now suppose that for some $A, B\in\reals$ we have $A~grad~T + B~grad~H = 0$. Then we must have the following
	\begin{align}
		A \left<\sigma(u), \alpha'( \frac{\partial\widetilde{u}}{\partial u} )\right> + B\frac{\partial\widetilde{u}}{\partial u} = 0.\\
		A \left<\widetilde{\sigma}(v), \alpha'( \frac{\partial\widetilde{v}}{\partial v} )\right> + B\frac{\partial\widetilde{v}}{\partial v} = 0.
	\end{align}
	Which gives the equation
	\begin{equation}\label{eq:omegarho}
		A \left<\sigma(u)-\widetilde{\sigma}(v), \alpha'( \frac{\partial\widetilde{u}}{\partial u}\frac{\partial\widetilde{v}}{\partial v} )\right> = 0.
	\end{equation}
	Note that since $\omega(u)\neq \rho(v)$ we have $\sigma(u)\neq \widetilde\sigma(v)$. And since $\sigma$ and $\widetilde\sigma$ are both of unit length and are incoming to $\partial S$, \cref{eq:omegarho} implies that $A=0$.
	But then $B=0$ and so $grad~T$ and $grad~H$ are linearly independent. Thus $G^{-1}(0,0)$ is a 0-dimensional submanifold of $U\times V$, that is $G^{-1}(0,0)$ is countable. Let $\widetilde U\times \widetilde V$ be the image of $U\times V$ in $(\T{S}_{x_0})^{2}$, let $\widetilde O$ be the image of $G^{-1}(0,0)$ in $\widetilde U\times\widetilde V$. We find corresponding $\widetilde U_{w}\times \widetilde V_{\rho}$ and $\widetilde O_{(\omega,\rho)}$ for each $(\omega,\rho)\in (\T{S}_{x_0})^{2}$ such that the geodesics generated by them have the same travelling times and endpoint. We can find a countable subcover of those  $\widetilde U_{w}\times \widetilde V_{\rho}$, giving a countable union of countable sets $\mathcal{O} = \cup \widetilde O_{(\omega,\rho)}$, which proves the lemma.
\end{proof}

\begin{proof}[Proof of \Cref{thm:nontrappedboundary}]
	Start with an open set $U\subseteq T_1\partial K$ and let \[X:[a,b]\to S_K,\] be the convex front constructed from $U$ as in \Cref{thm:convexlemma}. Let $\omega(u)$ be the unit outward normal at $X(u)$ and define $Y(u) = (X(u),\omega(u)) \in \uT S_K$. Also let $\Sigma = Y([a,b])$, and define $\Sigma'$ to be the set of non-tangent points in $\Sigma$. By \Cref{thm:onetangent} the set $\Sigma\backslash\Sigma'$ is countable. Now define $\widetilde\Sigma$ to be the set of trapped points in $\Sigma'$.
	
	Let $F = \{1,\cdots,k\}$ and define the space
	\begin{equation}
		\widetilde F = \prod_{i=1}^\infty F
	\end{equation}
	We give $\widetilde F$ the following metric
	\begin{equation}
		\rho(a,b) = \sum_{i=1}^{\infty}\frac{1}{3^i}\delta(a_i,b_i)
	\end{equation}
	Here $\delta(a_i,b_i)=0$ if $a_i = b_i$ and 1 otherwise. Now define a function $f:\widetilde\Sigma\to \widetilde F$ as follows. Given $x$ in $\widetilde\Sigma$, the geodesic generated by $x$ will hit reflect of the components of $K$ in a certain order, say $K_{i_1}, K_{i_2},\dots, K_{i_n}, \dots$. We define $f(x) = (i_1,i_2,\dots)$, where $i_n$ encodes the fact that the geodesic generated by $x$ will hit $K_{i_n}$ at the $n^{\text{th}}$ reflection. We first show that $f$ is continuous. Fix for a moment $x_0\in \widetilde\Sigma$. Then given $\varepsilon>0$ there exists an integer $N$ such that $(2\cdot 3^{N-1})^{-1}<\varepsilon$.
	Now let $(i_1,i_2,\dots,i_N,\dots) = f(x_0)$. Then by assumption $x_0$ will hit $K_{i_j}$ transversally for each $1\leq j\leq N$. So there is a neighbourhood $V$ of $x_0$ in $\widetilde\Sigma$ such that every $y$ in $V$ hits the components $K_{i_1},\dots,K_{i_N}$ in that order. So there is a $\delta>0$ such that if $d(x_0,y)<\delta$ we have $y\in V$. Now simply compute the distance in $\widetilde F$:
	\begin{equation}\label{eq:metricbound}
		\rho(f(x_0),f(y)) = \sum_{i=1}^{\infty}\frac{1}{3^i}\delta(f(x_0)_i,f(y)_i) \leq \sum_{i=n}^\infty \frac{1}{3^i}\leq \frac{1}{2\cdot 3^{N-1}}
	\end{equation}
	So we get that $\rho(f(x_0),f(y))<\varepsilon$ as required. Thus $f$ is continuous.
	
	We show that $f$ is injective as well.
	Suppose $x,y\in\Sigma$ are two nearby points such that the geodesic they generate will reflect off the same components $K_{\eta_1},\dots,K_{\eta_n}$ of $K$, in the same order $\eta_1,\dots,\eta_n$. By assumption, between $x$ and $y$ the convex front will propagate forward without tangencies (for the $n$ reflections). Let $\widetilde x$ and $\widetilde y$ be the points $x$ and $y$ after $n$ reflections. We now use the fact that the  generalised geodesic flow is uniformly hyperbolic (see \cite{MR1022522} and \cite{MR3808990}) to get some bounds on the distance between $x$ and $y$.
	
	Let 
	\begin{equation}
		b=\min_{1\leq i,j\leq k}\{\text{dist}(K_i,K_j)\}
	\end{equation}
	Here $k$ is the number of components of $K$. Also let $\widetilde\gamma:[0,\widetilde L]\to \uT S_K$ be a unit speed geodesic from  $\widetilde x$ to $\widetilde y$, where $\widetilde L = d(\widetilde x,\widetilde y)$. There exist $t_x,t_y\geq nb$ such that $\genflow_{t_x}(x) = \widetilde x$ and similarly for $y$. Let $t_0 = \text{max}\{t_x,t_y\}$. Then define $\gamma = \genflow_{-t_0}(\widetilde\gamma)$. Now consider the length
	\begin{equation}\label{eq:distbounds}
		L = \int_0^{\widetilde L} \norm{\pdiff{t	}\gamma}\ dt \leq \int_0^{\widetilde L} \norm{d_x\genflow_{t_0}}\cdot\norm{\pdiff{t}\widetilde \gamma}\ dt
	\end{equation}
	Since $\gamma$ is unit speed parameterised, we can use the hyperbolicity of $\genflow_t$ to get the following
	\begin{equation}
		d(x,y)\leq L \leq \widetilde L a\lambda^{t_0} \leq d(\widetilde x,\widetilde y)a\lambda^{nb}
	\end{equation}
	For some constants $a>0$ and $0<\lambda<1$. 
	Note that to get the last inequality we used the fact that $t_0\geq nb$.
	Thus $d(x, y)\to 0$ as $n\to\infty$. So we can see that for $x,y\in\widetilde\Sigma$ such that $f(x)=f(y)$, we have $x = y$. That is, $f$ is injective.
	
	So $f$ is bijective onto its image. Note that $d(\widetilde x,\widetilde y)$ is bounded above by some constant $C>0$. So \cref{eq:distbounds} gives $d(x,y)<Ca\lambda^{nb}$,
	where $n$ is the maximum number such that $f(x)_i = f(y)_i$ for all $1\leq i \leq n$. So given $\varepsilon>0$ there is some $N\in\mathbb{N}$ such that $Ca\lambda^{Nb}<\varepsilon$. Now by \cref{eq:metricbound}, setting \[\rho(f(x),f(y))<\frac{1}{3^{2N}}\]
	will necessarily force $f(x)_i = f(y)_i$ for all $1\leq i \leq N$. So we must have $d(x,y)<\varepsilon$. Thus the inverse of $f$ is also continuous, meaning it is a homeomorphism onto its image.
	
	It is well known that $\widetilde F$ has topological dimension zero (see \cite{MR0482697} for example). Thus $f(\widetilde\Sigma)$ will also have dimension 0, and since $f$ is a homeomorphism, so will $\widetilde\Sigma$. Therefore the dimension of $\Sigma\backslash(\widetilde\Sigma\cup\Sigma')$ is 1, so the result follows.
	\end{proof}

\section*{References}
\bibliographystyle{elsarticle-num}
\bibliography{bibliography}
\end{document}